\documentclass[utf8]{FrontiersinVancouver} 

\usepackage{url,hyperref,lineno,microtype,subcaption}
\usepackage[onehalfspacing]{setspace}

\def\keyFont{\fontsize{8}{11}\helveticabold }
\def\firstAuthorLast{D. Langemann  {et~al.}} 
\def\Authors{Dirk Langemann\,$^{1}$ and Olesia Zavarzina\,$^{2,*}$}
\def\Address{$^{1}$Institute of Partial Differential Equations, Technische Universit{\"a}t Braunschweig, Universit{\"a}tsplatz 2, 38106 Braunschweig, Germany\\
$^{2}$Department of Mathematics and Informatics, V.N. Karazin Kharkiv  National University, 61022 Kharkiv, Ukraine}
\def\corrAuthor{Olesia Zavarzina}

\def\corrEmail{olesia.zavarzina@yahoo.com}

\newtheorem{theo}{Theorem}[section]
 \newtheorem{definition}[theo]{Definition}
 \newtheorem{example}[theo]{Example}
  \newtheorem{lem}[theo]{Lemma}

\newtheorem{propos}[theo]{Proposition}
\newtheorem{rem}[theo]{Remark}
\newtheorem{prob}[theo]{Problem}
\newtheorem{cor}[theo]{Corollary}

\newcommand{\R}{{\mathbb{R}}}

\newcommand{\N}{{\mathbb{N}}}

\newcommand{\Z}{\mathbb{Z}}

\begin{document}
\onecolumn
\firstpage{1}

\title[Plasticity on the real line]{Expand-contract plasticity on the real line} 

\author[\firstAuthorLast ]{\Authors} 
\address{} 
\correspondance{} 

\extraAuth{}

\maketitle

\begin{abstract}

The article deals with plastic and non-plastic sub-spaces $A$ of the real line $\R$ with the usual Euclidean metric $d$. It investigates non-expansive bijections, proves properties of such maps and demonstrates their relevance by hands of examples.
Finally, it is shown that the plasticity property of a sub-space $A$ contains at least two complementary questions, a purely geometric and a topological one. Both contribute essential aspects to the plasticity property and get more critical in higher dimensions and more abstract metric spaces.

\tiny
 \keyFont{ \section{Keywords:} metric space, non-expansive map, plastic space, expand-contract plasticity, Banach space} 
\end{abstract}

\section{Introduction}

Here, we investigate properties of plastic metric spaces. Shortly speaking, a metric space is plastic if every non-expansive bijection is an isometry, cf.\ Sec.~\ref{secbasic}.

We will observe that the plasticity property consists of a geometrical sub-problem and a topological sub-problem. That is the reason why plasticity of
a metric space, which can be easily defined, evolves as a challenging mathematical problem. In particular, we observe that the plasticity of a metric space
is not inherited from sup-spaces, i.\,e.\ from including spaces, and it does not inherit to sub-spaces, i.\,e.\ to included spaces.

In this article, we concentrate on metric spaces which are sub-spaces of the real axis, and already in this apparently simple situation, the typical difficulties come to the light.

The probably first work devoted to the plasticity problem is \citep{FHuDeh}, however, the term "plasticity" appeared much later and the problem itself remained unnoticed for several decades. A short literature survey and the information about the current progress in solution of the problem is found in Sec.~\ref{secplastic}

The paper is organized as follows. Sec.~\ref{secbasic} introduces the basic concepts and illustrates the existence of non-expansive bijections in the
case that the metric space is a union of closed intervals. This case demonstrates the geometrical aspects of the problem. Then, Sec.~\ref{secresults}
discusses the plasticity of metric spaces by means of metric spaces which are unbounded sequences of points, investigates the relevance of accumulation points, continuous subsets and attacks the more topological parts of the plasticity concept. Finally, Sec.~\ref{secconclu} resumes the observations and gives a short outlook to further research.

\section{Basic concepts}\label{secbasic}

We denote a metric space by $(A,d)$ where $A$ is the set of points and $d\,:\,A\times A\to\R_{+}=\{x\in \R\colon x\ge 0\}$ is the distance obeying the known axioms of
positivity, symmetry, non-degeneracy and the triangle inequality.

\subsection{Non-expansive maps}\label{secnonexp}

A map $\varphi\,:\,A\to A$ from the metric space $A$ into itself is called non-expansive if
\begin{equation}\label{eq01}
d(\varphi(x),\varphi(y))\le d(x,y)\;\mbox{ for all }\;x,y\in A
\end{equation}
is fulfilled. If the equality holds for all pairs $x,y\in A$, then $\varphi$ is an isometry.

The condition in Eq.~(\ref{eq01}) is equivalent to the Lipschitz-continuity of the map $\varphi$ on $A$ with Lipschitz constant $1$.
Thus, a non-expansive $\varphi$ is also continuous on $A$.

We will investigate metric spaces $A\subseteq A_{\mathrm{ex}}$ which are embedded in a metric sup-space $(A_{\mathrm{ex}},d_{\mathrm{ex}})$, e.\,g.\
because the space $A_{\mathrm{ex}}$ might be known and well understood, and thus its points or rather a selection of them serve as elements of $A$.
Now, it is obvious that the restriction of the metric space $(A_{\mathrm{ex}},d_{\mathrm{ex}})$ to the set $A$ leads to the metric space $(A,d)$ by the
restriction of the distance $d=d_{\mathrm{ex}}|_{A\times A}$ to the set $A$. It is less obvious whether a metric space $(A,d)$ can be extended to
a sup-set $A_{\mathrm{ex}}$ by choosing an appropriate $d_{\mathrm{ex}}$. But it is always possible, to choose a function
$\hat{d}_{\mathrm{ex}}\,:\,A_{\mathrm{ex}}\times A_{\mathrm{ex}}\to \R_+$, which fulfills the properties of symmetry, non-degeneracy and positivity with $\hat{d}_{\mathrm{ex}}|_{A\times A}=d$, which of course is not a metric in general. Then, we can
define the metric
\[
\widetilde{d}_{\mathrm{ex}}(x,y)=\inf\limits_{n,\{z_0,\ldots,z_n\}}\left[\hat{d}_{\mathrm{ex}}(x,z_0)+\sum\limits_{i=0}^{n-1} \hat{d}_{\mathrm{ex}}(z_i,z_{i+1})
+\hat{d}_{\mathrm{ex}}(z_n,y)\right]
\]
as the infimum over all possible paths of arbitrary length between $x$ and $y$. However, such a metric $\widetilde{d}_{\mathrm{ex}}$ may not really be
an extension. As in the real life, if one builds a new paths, which are shorter, the old ones may no longer be used. In our notation, this means
that it may happen $\widetilde{d}_{\mathrm{ex}}(x,y)<d(x,y)$ for some $x,y\in A$. 

Nevertheless, one may define a real extension $d_{\mathrm{ex}}$ of the metric $d$, which is more artificial and a bit similar to the French railways metric in the following way. Let us fix a point  $x_0$ of the set $A$ and define an arbitrary metric $d_{A_{\mathrm{ex}}}$ on the set 
$(A_{\mathrm{ex}}\setminus A) \cup \{x_0\}$, which might be the discrete metric or any other metric. Although less intuitive, the needed extension is
\[
d_{\mathrm{ex}}=
\begin{cases}
    d(x,y), & \mbox{for } x,y \in A;\\
    d_{A_{\mathrm{ex}}}(x,y), & \mbox{for } x,y \in (A_{\mathrm{ex}}\setminus A) \cup \{x_0\};\\
    d(x,x_0)+d_{A_{\mathrm{ex}}}(x_0,y) & \mbox{for } x\in A, y \in (A_{\mathrm{ex}}\setminus A).
\end{cases}
\]
existing and easily available. Therefore, 
we will not distinguish between $d_{\mathrm{ex}}$ and $d$ in the following, but
use the distance $d$ in the extended metric space as well as in the sub-space.

Oppositely, it is not evident whether the existence of a non-expansive map $\varphi_{\mathrm{ex}}\,:\,A_{\mathrm{ex}}\to A_{\mathrm{ex}}$ provides a non-expansive map $\varphi\,:\,A\to A$ because
the simple restriction $\varphi=\varphi_{\mathrm{ex}}|A$, although still Lipschitz continuous, is not 
necessarily a map into $A$. It might happen that the image $\operatorname{im}\varphi=\varphi(A)\subseteq A_{\mathrm{ex}}$ is 
not a subset of $A$.
The opposite question whether a non-expansive $\varphi\,:\,A\to A$ can be extended to a non-expansive 
map on the extended space $A_{\mathrm{ex}}$ is the question about thr extension of Lipschitz maps preserving the Lipschitz constant. In particular, it is always possible for real-valued functions according to McShane's extension theorem \citep{McSh}. For functions from a subset of $\R^n$ to $\R^n$, the extension to the whole Euclidean space is possible due to Kirszbraun’s theorem \citep{Kir}.
We will see that non-expansive maps pose a lot of interesting questions and some of them can be answered, too.

\subsection{Plastic metric spaces}\label{secplastic}

Let us define a plastic metric space.

\begin{definition}\label{defplastic}
A metric space 
A is called expand-contract plastic (EC-plastic) -- or just plastic -- if every bijective non-expansive map
$\varphi\,:\,A\to A$ is an isometry.
 \end{definition}

Def.~\ref{defplastic} defines a plastic metric space $A$ via the non-existence of any non-expansive bijection of the metric space $A$ to itself, which is not
an isometry. Some simple examples are the non-plastic metric space $A=\R$ with the non-isometric non-expansive bijective map $\varphi\,:\,x\mapsto x/2$
and the plastic metric space $A=[0,1]\subset\R$ with exactly the two non-expansive bijections $\varphi_1=id.$ and $\varphi_2\,:\,x\mapsto 1-x$, which are both isometries.

The only general result concerning plasticity of metric spaces states that every totally bounded metric space is plastic, s.\ \citep{NaiPioWing} for details. In fact, in \citep{NaiPioWing} a more general result was obtained, i.\,e.\ so called strong plasticity of totally bounded metric spaces was shown.

\begin{definition} \label{def-str-plast}  A metric space $A$ is said to be strongly plastic if for every mapping $\varphi\,:\, A \to A$
the existence of points $x, y \in A$ with  $d(\varphi(x),\varphi(y)) > d(x,y)$ implies the existence of two points $\tilde{x}, \tilde{y}\in A$ for which $d(\varphi(\tilde{x}),\varphi(\tilde{y})) < d(\tilde{x},\tilde{y})$ holds true.
\end{definition}

This property and its uniform version were studied in \citep{KZ2023}. It says that any expansion of a distance between two points implies
the existence of two other points which are contracted by the map $\varphi$. Observe it is extremely important not to interchange expansion and contraction.

In \citep{CKOW2016} the following intriguing question was posed.
 \begin{prob}
 Is it true, that the unit ball of an arbitrary Banach space is  plastic?
 \end{prob}

 Observe, that in finite dimensions this question is answered positively since in finite dimensions, the unit ball is compact and thus
 totally bounded. So the question is open only in the infinite dimensional case as well as the following more general problem.

 \begin{prob}
 For which pairs $(X,Y)$ of Banach spaces, every bijective non-expansive  map $\varphi\,:\, B_X(0) \to B_Y(0)$
 between the unit balls is an isometry?
 \end{prob}

 There is a number of relatively recent particular results, devoted to these problems, s.~\citep{AnKaZa}, \citep{HLZ}, \citep{KZ2017}, \citep{Leo} and \citep{Zav}.
 There exists also a circle of problems connected with plasticity property of the unit balls. In \citep{KarZav} and \citep{Zav2}, the so called linear expand-contract plasticity of ellipsoids in separable Hilbert spaces was studied. That means that only the linear non-expansive bijections were considered in the definition of plasticity.

Many natural questions concerning plasticity  seem to have no answer or even have not yet been considered. 
In 2020, E. Behrends \citep{Behr} draw attention to the fact that nobody studied the subsets of the real line with respect to the plasticity problem. He tried to attack this problem and received some results in this direction, however, decided not to publish them. So, the following problem is still open.

\begin{prob}
What characterises plastic sub-spaces of the real line $\R$ with the usual metric $d$?
\end{prob}

In spite of the seeming simplicity of the question, it is not so easy to deal with. Let us first list the previously known results.
As we mentioned before, the set $\R$ itself with the usual metric is not plastic. If one considers any bounded subset, it is already plastic due to its total boundedness.

On the other hand it is easy to show that the set of integers $\Z$ with the same usual metric is plastic in spite of its unboundedness, as well as the set $\R\setminus\Z$. The proof of the plasticity of both mentioned spaces may be found in \citep{NaiPioWing}. In the proof of plasticity of the set $\R\setminus\Z$, one of the possible cases was missed, nevertheless the statement is still correct.

Already, these examples show that there is no simple answer to the question whether a metric space is plastic or not. Rather we could give
the interpretation that there are some critical points, like e.\,g.\ the integers in these examples, which every non-expansive bijection $\varphi$ definitely
has to pass, what relates to the geometry of the metric space $A$, and that there are some parts of the metric space which cannot be glued to each other
like singular points or open intervals, what relates to the topological aspects of plasticity. We see that sub-spaces of the real axis are already
sufficiently multifaceted to study the plasticity problem of metric spaces. The question whether
more general metric spaces are plastic, provokes analogous difficulties and again contains geometrical and topological aspects.

Here, we will generalize the
known results and say something more about plastic sub-spaces of the real line. The previously mentioned results explain why we consider only unbounded sets in what follows.

All over the text, we use the notion $d$ for the usual Euclidean metric $d(x,y)=|x-y|$ for $x,y\in\R$. Round brackets denote open intervals $(x,y)=\{z\in\R\,:\,x<z<y\}$, and square brackets denote closed intervals $[x,y]=\{z\in\R\,:\,x\le z\le y\}$.  


\subsection{A subset of the real axis}\label{asubsetofr}

We have seen that the real axis $\R$ has sufficiently interesting metric sub-spaces for the investigation of plasticity. The Lipschitz condition
in Eq.~(\ref{eq01}) lets us easily decide whether a map $\varphi\,:\,\R\to\R$ is non-expansive or not -- just by the graph of the map $\varphi$, s.\ Fig.~\ref{fig1}. Due to our
considerations in Sec.~\ref{secnonexp}, which is applied here with $A$ as union of intervals and $A_{\mathrm{ex}}=\R$, the map $\varphi$ can be extended -- not necessarily in a unique manner -- as non-expansive function 
$\varphi_{\mathrm{ex}}$ on the entire axis $\R$. Thus $\varphi_{\mathrm{ex}}$ is continuous on $\R$, too.


\begin{figure} 
\begin{center}
\includegraphics[width=7cm]{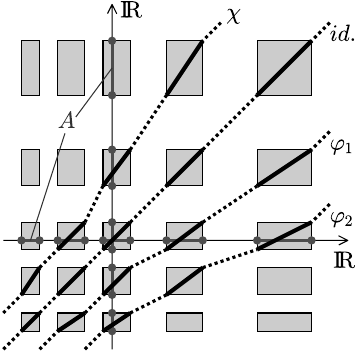}
\end{center}
\caption{Non-expansive maps $\varphi_1$, $\varphi_2$ and $id.$ and an expansive map $\chi$ for a union $A\subset\R$ of closed intervals
of increasing length. The Cartesian product $A\times A$ is given in grey, and the bijections are black.}
\label{fig1}
\end{figure}

Fig.~\ref{fig1} shows examples of bijective maps from the union of intervals
$A=\ldots\cup[a_{2},a_{3}]\cup [a_{4},a_{5}]\cup\ldots\subset\R$ onto itself. In this example, the closed
interval and the interspaces have increasing lengths, in detail $a_{\ell+1}-a_{\ell}\ge a_{\ell-1}-a_{\ell-2}$ for all $\ell\in \Z$. Due to its continuity,
every bijection $\varphi$ passes monotonically a rectangle in $A\times A$. 
In this example with increasing lengths of the respective intervals, we easily detect particular extensions $\varphi_{\mathrm{ex}}\,:\,\R\to\R$ with
$\varphi_{\mathrm{ex}}|_A=\varphi$ and a slope bounded by $1$, because the endpoints of the interspace could be used in Eq.~(\ref{eq01}). 
Hence, the functions $id.$ and $\varphi_i,\, i=1,2$ below the diagonal are non-expansive, and the function $\chi$ above the diagonal is expansive.

\section {Main results}\label{secresults}

Let us start with 
some interesting observations on simple situations of $A$, e.\,g.\ some sets of singular points.

\begin{propos}\label{prop_easy}
  Let $A =\{a_i\}_{i=-\infty}^{+\infty}\subset\R$ be an increasing sequence that obeys
  \begin{equation}\label{eq03}
  d(a_{i-1},a_i)\leq d(a_i,a_{i+1})\;\mbox{ for all }\; i\in \Z 
  \end{equation}
 and
 \begin{equation}\label{eq04}
 d(a_{j-1},a_j)< d(a_j,a_{j+1})\;\mbox{ for at least one }\; i\in \Z.
 \end{equation}
 Then $(A,d)$ is not plastic.
\end{propos}
\begin{proof}
The shift $\varphi\,:\,a_i\mapsto a_{i-1}$ is an example of a non-expansive bijection which is not an isometry.
\end{proof}

\begin{rem} The relation sign in Eqs.~(\ref{eq03}) and (\ref{eq04}) might be commonly inverted so that the distances between two subsequent
points of $A$ decrease instead of increase, and the statement remains unchanged.
\end{rem}

Further let us consider sets which are bounded from one side. We will proceed with the following lemma.
\begin{lem}\label{lem}
Let $A\subset\R$ be a set without accumulation points which is bounded from one side. Let $a$ be a minimal -- or maximal -- element
of $A$ and $\varphi\,:\, A\to A$ be a bijective non-expansive map. Then $\varphi(a)=a$.
\end{lem}

\begin{proof}
Without loss of generality we may consider the case when $a$ is a minimum.
  Assume $\varphi(a)\neq a$. Then there is $b\in A$ such that $\varphi(b)=a$.

Claim: \textit{Let be $c\in A$. Then $c\leq b$ implies $\varphi^n(c)\leq b$ for every $n\in \N$.}

{\it Proof of the Claim}. We will use the induction in $n$.
Indeed, if $\varphi^n(c)\leq b$ and $\varphi^{n+1}(c)> b$ we have
  $$d(\varphi^n(c),b)\geq d(\varphi^{n+1}(c),a)  > d(b,a).$$
   {\it This contradiction completes the proof of the Claim.}
   
   Since  $$d(a,b)\geq d(\varphi(a),\varphi(b))=d(\varphi(a),a),$$
  we have $\varphi(a)\leq b$. Thus the Claim provides $\varphi^n(a)\leq b$ for every $n\in \N$.
   Now, the segment $[a,b]$ is a trap for those points, which were mapped there. Our aim is to find such a ``trapped'' point out of the interior of the segment $[a,b]$ and show that this leads to a contradiction. There are only two possible cases.
   
{\bf Case 1:} $\varphi(a)= b$. In this case points $a$ and $b$ were swapped by $\varphi$. Then such a ``trapped'' point is the closest from the right-hand side point to $b$. That means, there is $c>b$ such that $d(b,c)<d(b,d$) for any $d>b$. Such point $c$ exists since $A$ is unbounded from above and there is no accumulation points. The point $c$ cannot be mapped outside the segment $[a,b]$ since it gives the contradiction with non-expansiveness of $\varphi$.

{\bf Case 2:} $\varphi(a)< b$. With such a condition a ``trapped'' point is $\varphi(a)$ itself. 

In both cases we have a point $t$ which does not belong to the interior of the segment $[a,b]$ such that $\varphi(t)$ belongs to this interior. Consider an orbit of this point $t$, i.\,e.\ the set $\{\varphi^n(t)\}_{n=1}^\infty$. Due to the bijectivity of $\varphi$ this orbit does not have repeating elements. Thus we have obtained a bounded infinite subset in $A$ which contradicts the fact that $A$ does not have accumulation points.
\end{proof}

\begin{rem}
The condition about absence of accumulation points in Lemma~\ref{lem} cannot be omitted.
\end{rem}
This remark is confirmed by the following example.
\begin{example}\label{example1}
Let $A=\Z_+\cup Q$, where $Q=\{\frac{1}{4}+\frac{1}{n}, n\geq 4\}$. The bijective non-expansive map $\varphi$ is
$$
\varphi(a)=
\begin{cases}
 a-1, & \mbox{for } a\in \N, \\[0.3ex]
\frac{1}{2}, & \mbox{for } a=0, \\[0.3ex]
 \frac{1}{4}+\frac{1}{n+1}, & \mbox{for } a=\frac{1}{4}+\frac{1}{n}\in Q.
\end{cases}
$$
We see that $\varphi$ is bijective and it does not save the minimal element of $A$.
We check that it is non-expansive.
\begin{enumerate}
  \item For all $a,b\in \N$, the isometry $d(\varphi(a),\varphi(b))=d(a,b)$ is valid.
  \item For $a\in \N$, $b=0$, it holds $d(\varphi(a),\varphi(b))=|a-\frac{3}{2}|<a=d(a,b)$.
  \item For $a\in \N$, $b=\frac{1}{4}+\frac{1}{n}\in Q$, we have $d(\varphi(a),\varphi(b))=|a-\frac{5}{4}-\frac{1}{n+1}|<|a-\frac{1}{4}-\frac{1}{n}|=d(a,b)$.
  \item For $a=0$, $b=\frac{1}{4}+\frac{1}{n}\in Q$, it holds $d(\varphi(a),\varphi(b))=|\frac{1}{4}-\frac{1}{n+1}|<|\frac{1}{4}+\frac{1}{n}|=d(a,b)$.
  \item In the case $a=\frac{1}{4}+\frac{1}{n}\in Q$, $b=\frac{1}{4}+\frac{1}{m}\in Q$, without loss of generality we may assume $n<m$. Then 
  \[
  d(\varphi(a),\varphi(b))=\frac{1}{n+1}-\frac{1}{m+1}<\frac{1}{n}-\frac{1}{m}=d(a,b).
  \]
\end{enumerate}
\end{example}
The described set is depicted on the Fig.~\ref{fig2}, left.

Lemma~\ref{lem} immediately implies the following corollary.
\begin{cor}\label{cor}
Let $A\subset\R$ be an unbounded set without accumulation points. Let $A$ have a minimal or maximal element and let $\varphi\,:\,A\to A$ be a bijective non-expansive map. Then $\varphi$ is an isometry, moreover, the identity.
\end{cor}
\begin{proof}
Without loss of generality, we may consider the case when $a$ is a minimum.
Let us show that $\varphi(x)=x$ for every $x\in A$. Indeed, for the minimal element $a$ Lemma~\ref{lem} ensures that $\varphi(a)=a$. Now suppose for some fixed $y\in A$ the condition $\varphi(x)=x$ holds for every $x<y, x\in A$. Consider 
\[
A_1=A\setminus\Big\{\bigcup_{x\in A, x<y}\{x\}\Big\}. 
\]
Then $\varphi|_{A_1}\colon A_1 \to A_1$ is a bijective non-expansive map, and $y$ is a minimal element. Then $\varphi(y)=y$ due to Lemma~\ref{lem}.
\end{proof}
Proposition 4.1 in \citep{NaiPioWing} states that for convex (in the sense of the same article) metric spaces hereditarily EC-plasticity implies boundedness. Moreover, for convex subsets in Euclidean $\R^n$ hereditarily EC-plasticity and boundedness are equivalent. However, the authors note that convexity is a too strong condition.

  In \citep{NaiPioWing}, Theorem 4.3 states that an unbounded metric space with at least one accumulation point contains a non-plastic space. Corollary~\ref{cor} demonstrates that the presence of an accumulation point is essential in the mentioned theorem, since it allows to build examples of unbounded hereditarily plastic spaces. 
  
Let us go back to Example~\ref{example1} and remark another interesting property of non-expansive bijections on $\R$. Suppose we have a set $A\subset\R$ and a function $\varphi\colon A\to A$.  We will say that $\varphi$ preserves the relation ``between'' on the set $A$ if for any $x,y,z \in A$ with $x<y<z$ we have $\varphi(x)<\varphi(y)<\varphi(z)$.  Example~\ref{example1} shows that non-expansive bijections do not have to preserve the relation ``between''. Surprisingly, there is an example demonstrating the same property with a set without any accumulation points.
\begin{example}\label{example2}
Let $A=\N\cup Q$, where $Q=\{2k, k\in \Z_-\}$. The bijective non-expansive map $\varphi$ is defined by
$$
\varphi(a)=
\begin{cases}
 a+6, & \mbox{if } a\leq-4, \\
  a+3, & \mbox{otherwise}.
\end{cases}
$$
The map $\varphi$ does not preserve the relation ``between'' since $-4<-2<0$ but $\varphi(-2)<\varphi(-4)<\varphi(0)$.
Let us check that $\varphi$ is non-expansive.

\begin{enumerate}
  \item If both $a,b\geq-2$ or both $a,b\leq-4$, the non-expansiveness of $\varphi$ is obvious.
  \item If $a\geq-2$ and $b\leq-4$, then $d(\varphi(a),\varphi(b))=|a-b-3|\leq |a-b|=d(a,b)$. Only for $a=-2$ and $b=-4$, the inequality $a-b<3$ is valid, but even in this case, the previous inequality holds true.
\end{enumerate}
\end{example}
The described set is depicted on the Fig.~\ref{fig2}, right.
\begin{figure}
    \centering
        \includegraphics[width=6cm]{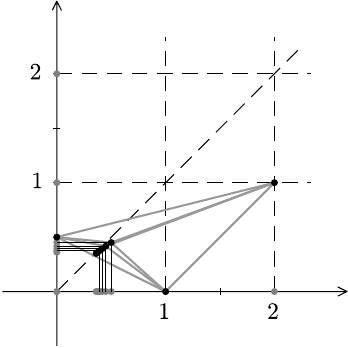}
            \hspace*{2em}
            \includegraphics[width=6cm]{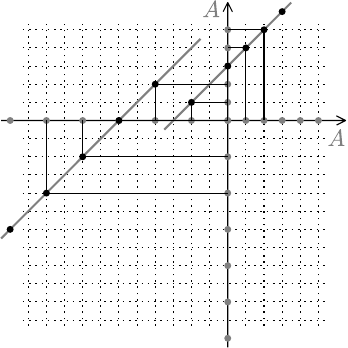}

    \caption{Left: Illustration of Exp.~\ref{example1}. Right: and of Exp.~\ref{example2}. The grey dots on the axes indicate $A$. The black dots mark the respective bijection. Clearly, no connection of two points has a slope larger than $1$.}
    \label{fig2}
\end{figure}

Further we are going to present a sufficient condition for a set in $\R$ to be plastic.
Let us introduce the set
$$D_A =\{p\in\R \colon p=d(a,b) \text{ for some } a,b \in A \text{ with } [a,b]\cap A=\{a,b\}\}.$$
Obviously, several pairs of points may be situated on the same distance. That is why for every $p\in D_A$ we call its multiplicity the number of pairs of points in $A$ which are on the distance $p$. This multiplicity may be finite or infinite.
\begin{theo}\label{theo38}
Let $A\subset\R$  have no accumulation points and  let $D_A$ have a maximum of finite multiplicity or a minimum of finite multiplicity. Then $(A,d)$ is a plastic metric space.
\end{theo}
\begin{proof}
Without loss of generality, we may assume that $D_A$ has a minimum $a\in \R$ of finite multiplicity $k\in \N$. Let us denote
$$X_a= \{x_n\in A, n=1,\dots,2k,  d(x_i,x_{i+1})=a, i=1,3,\dots,2k-1 \}.$$
Let us take $x_i\leq x_j$ for all $i,j$ with $1\le i<j\le 2k$. 
Consider an arbitrary non-expansive bijection $\varphi\,:\,A\to A$. Due to the non-expansiveness of $\varphi$, we may conclude that $\varphi$ maps $X_a$ bijectively onto itself. Thus $\varphi|_{X_a}$ is an isometry on $X_a$. In particular, we find $d(x_1, x_{2k})=d(\varphi(x_1),\varphi(x_{2k}))$. Since this distance is the biggest one on $X_a$, either $\varphi(x_1)=x_1$ and $\varphi(x_{2k})=x_{2k}$ or $\varphi(x_1)=x_{2k}$ and $\varphi(x_{2k})=x_1$. We will refer them as cases 1 and 2, respectively. In the first case, obviously, for every $x\in A$ with $x_1<x<x_{2k}$, we get $\varphi(x)=x$, so, in this case $\varphi|_{[x_1,x_{2k}]\cap A}$ is the identity. In the second case, if the structure of $A$ allows it, $\varphi|_{[x_1,x_{2k}]\cap A}$ is the inversion, called total symmetry. Further, following the similar procedure as in Lemma~\ref{lem} we have that in the first case $\varphi$ is the identity, in the second case $\varphi$ is the total symmetry.
\end{proof}
\begin{rem}\label{rem39}
The conditions of Theorem~\ref{theo38} are sufficient but not necessary for the plasticity of a set without accumulation points.
\end{rem}
To make sure that the previous Remark~\ref{rem39} is true, one may consider the space $(\Z,d)$. For $D_\Z$ the minimum and the maximum are equal to 1 and have infinite multiplicity, but the space is plastic. However, we constructed the next example, which is less trivial, to show that plastic spaces which do not satisfy the condition of the previous theorem may have richer structure.
\begin{example}
Let $A=\{a_i\}_{i=-\infty}^{i=\infty} \subset \R$, where $\{a_i\}_{i=-\infty}^{i=\infty}$ is an increasing sequence such that
$$
d(a_i,a_{i+1})=
\begin{cases}
  |k|+1, & \mbox{for } i=2k, k\in \Z, \\[0.3ex]
  \frac{1}{k+1}, & \mbox{for } i=2k-1, k\in \N, \\[0.3ex]
  \frac{1}{|k|+2}, & \mbox{for } i=2k-1, k\in \Z_{-}.
\end{cases}
$$
The corresponding $D_A$ has no minimal or maximal element. However, $(A,d)$ is plastic. In fact, let $\varphi\,:\,A\to A $ be a non-expansive bijection. Then
$$d(\varphi(a_0),\varphi(a_1))\leq d(a_0,a_1)=1.$$
Suppose $d(\varphi(a_0),\varphi(a_1))=\frac{1}{n}$, where $n\geq 2$. Consider the open ball with the radius $n-1$ centered in $\varphi(a_0)$. Due to the structure of $A$, this ball contains only the point $\varphi(a_1)$ except for the centre. On the other hand, the open ball with the radius $n-1$ centered in $a_0$ for $n\geq 3$ contains more than two points, and for $n=2$, it contains two points but does not contain $a_1$. In both cases, we have a contradiction to the non-expansiveness of the map $\varphi$. That is why the only possible option is $$d(\varphi(a_0),\varphi(a_1))=d(a_0,a_1)=1.$$
Further, just in the same way as in Theorem~\ref{theo38} we have that $\varphi$ is either the identity or the inversion.
\end{example}
Now let us speak about the subsets which contain a continuous part. One may prove the following statement in the same way as the Proposition \ref{prop_easy}.
\begin{propos}\label{prop_int_np}
  Let be 
  \[
  A =\bigcup_{i=-\infty}^{+\infty}(a_i, b_i)\subset \R,
  \]
  where $ b_i<a_{i+1}$ be such a sequence of intervals that \begin{equation}\label{eqn1}
  d(a_{i},b_i)\leq d(a_{i+1},b_{i+1})
  \end{equation}
    and 
    \begin{equation}\label{eqn2}
    d(b_{i-1},a_{i})\leq d(b_{i},a_{i+1})
    \end{equation}
    for all $i\in \Z$. Furthermore, there exists $j\in \Z$ such that 
    \begin{equation}\label{eqn3}
    d(a_j,b_j)< d(a_{j+1},b_{j+1}) \text{ or } d(b_{i-1},a_{i})< d(b_{i},a_{i+1}).
    \end{equation}
     Then $(A,d)$ is not plastic.
\end{propos}
\begin{rem} In the same way as in Proposition \ref{prop_easy} the relation signs in Eqs.~(\ref{eqn1}), (\ref{eqn2})  and (\ref{eqn3}) might be commonly inverted. 
\end{rem}

Here is one more observation.
\begin{propos}
  Let $A\subset\R$ contain an interval $(a, +\infty)$ or $(-\infty, a)$. Then $(A,d)$ is not plastic.
\end{propos}
\begin{proof}
Without loss of generality, we discuss the case with $(a, +\infty)$. Let us define the map $\varphi$ with
$$
\varphi(x)=
\begin{cases}
  \varphi(x)=x, & \mbox{if } x\notin (a, +\infty),  \\[0.3ex]
  \varphi(x)=\frac{x+a}{2}, & \mbox{otherwise}.
\end{cases}
$$
This map is non-expansive, bijective and -- at the same time -- 
not an isometry.
\end{proof}
In \citep{NaiPioWing}, Theorem 3.9 shows the plasticity of the space $\R\setminus\Z$. Unfortunately, in the proof misses the case that the non-expansive bijection is a symmetry. However, the statement itself is true.
One may use the same reasoning to prove the next proposition.
\begin{propos}\label{prop_int_p}
  Let 
  \[
  A=\bigcup_{i=-\infty}^{+\infty}(a_i, b_i)\subset \R,
  \]
  where \[
  d(a_{i},b_i)=d(a_{i+1},b_{i+1})
  \;\mbox{ and }\;
  d(b_{i},a_{i+1})= d(b_{i-1},a_i).
  \] Then $(A,d)$ is plastic.
\end{propos}
\begin{rem}
Propositions \ref{prop_int_np} and \ref{prop_int_p} hold true with the closed intervals as well.
\end{rem}
\begin{rem}
On the other hand,  if we consider in the statement of Proposition \ref{prop_int_p} half-intervals 
\[
A=\bigcup_{i=-\infty}^{+\infty}[a_i, b_i)\subset \R
\;\mbox{ or }\;
A=\bigcup_{i=-\infty}^{+\infty}(a_i, b_i]\subset \R
\]
lead already to a non-plastic space.
\end{rem}
\begin{rem}
If we consider in the same statement, the set of the form 
\[
A=\bigcup_{i=-\infty}^{n}[a_i, b_i]\cup\bigcup_{i=n+1}^{+\infty}(a_i, b_i]\subset \R,\;\mbox{ where }\; n\in \N, 
\]
is also a non-plastic space.
\end{rem}
Fig.~\ref{fig3} illustrates the previous remark.
\begin{figure}[h!]
\begin{center}
\includegraphics[width=6cm]{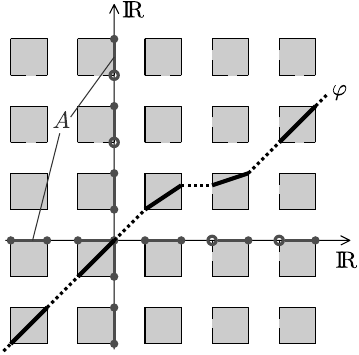}
\end{center}
\caption{Oppositely to Fig.~\ref{fig1}, half-open intervals allow that $\varphi$ does not pass entire rectangles in $A\times A$.
Rather, it might jump where the intervals can be glued to each other. Remark that this example contains a first half-open interval and all the following
intervals are half-open, too, cf.\ bijectivity. The topological properties of the intervals in $A$ enter the plasticity problem.}
\label{fig3}
\end{figure}
The reader easily provides more examples  which consist of open or closed intervals together with half-intervals, all of the same lengths. Again, we remark that the end-points of the intervals are critical points for the plasticity property.

\section{Conclusion}\label{secconclu}

The analysis of plastic sub-spaces $A$ of the real line $\R$ has shown that first, the Lipschitz-continuity of the map $\varphi\,:\,A\to A$ with Lipschitz constant $1$ leads to useful and instructive illustrations of the non-expansivity of the map $\varphi$, to which it is identical.

The plasticity property of a metric space turned out to contain two complementary aspects, a purely geometrical one and a topological one. Already on the real line $\R$, the different natures of both aspects become visible. Whereas the geometrical aspects is an extension of the non-expansivity of $\varphi$ on a simply-connected interval, the topological aspect leads to the question whether two or more sub-intervals can be glued at critical points by piecewise translations. Therefore, the investigation of sub-spaces of the real line $\R$ gives an appropriate framework for the investigation of the plasticity of metric spaces.

We expect that the interplay between the two natures of the problem gets more severe in higher dimensions. Already unions of rectangles and cuboids as sub-spaces of the $d$-dimensional Euclidean space $\R^d$ give a tremendous multiplicity of open, half-open and closed edges and sides -- complete or partial.

The named interplay between geometry and topology of the metric spaces gets more and more complicated and less intuitive the more abstract and the more elaborated the metric spaces are. We do not expect any clarification for e.\,g.\ metric spaces of functions before sub-spaces of the Euclidean spaces are understood.

Future research will concentrate on the question, what else can be said about plastic and non-plastic sub-spaces of the space $(\R,d)$. Furthermore, we will explore the extension of a metric space $A$ to larger sets in $A_{\mathrm{ex}}$ which contain $A$. In particular, the metric hull, i.\,e.\ the set  
\begin{equation}\label{eqmetricspan}
\operatorname{hull}_{A_{\mathrm{ex}}}(A)=\left\{x\in A_{\mathrm{ex}} \,:\,\exists y,z\in A\,: d(y,z)=d(y,x)+d(x,z)\right\}\subseteq A_{\mathrm{ex}},
\end{equation}
gives interesting perspectives in the context of the plasticity problem for the specification $A_{\mathrm{ex}}=\R$. We conjecture that the metric hull is the smallest proper extension of the metric space, which is simply connected in $A_{\mathrm{ex}}$ and where the plasticity is dominated by the geometry. There, the topology might be sub-ordinated. In the medium term, we hope for insight in the question how geometry and topology interact in the plasticity of a metric space.

\section*{Conflict of Interest Statement}

The authors declare that the research was conducted in the absence of any commercial or financial relationships that could be construed as a potential conflict of interest.

\section*{Author Contributions}

XX automatically generated after submission

DL writing -- review \& editing, methodology, interpretation

OZ writing -- review \& editing, conceptualization, formal analysis, project administration

First verbal form: The second author conveived the presented ideas and developed the theoretical results. The first author contributed the interpretation and illustration. Both authors discussed the results and contributed to the final manuscript. Both authors fully agree with the content of the article.


\section*{Funding}
 The research was partially supported by the Volkswagen Foundation grant within the frameworks of the international project ``From Modeling and Analysis to Approximation''. The second author was also partially supported by Akhiezer Foundation grant, 2023.

\section*{Acknowledgements}
The authors are grateful to Vladimir Kadets for valuable remarks and for pointing us the results about the extension of Lipschitz maps. We are also thankful to Ehrhard Behrends for drawing our attention to the problem considered in this article.

\bibliographystyle{Frontiers-Harvard} 




\end{document}